\documentclass[reqno]{amsart} 
\usepackage{amsfonts, amsmath, amsthm, amssymb, latexsym, graphicx}

\theoremstyle{plain}
\newtheorem{theorem}{Theorem}[section]
\newtheorem{corollary}[theorem]{Corollary}
\newtheorem{lemma}[theorem]{Lemma}

\theoremstyle{definition}

\newtheorem{remark}[theorem]{Remark}

\newtheorem{example}[theorem]{Example}

\numberwithin{equation}{section}

\title[ODEs associated with CMS models]{The differential equations associated with Calogero-Moser-Sutherland particle models 
in the freezing regime}

\author{Michael Voit, Jeannette H.C. Woerner} 
\address{Fakult\"at Mathematik, Technische Universit\"at Dortmund,
          Vogelpothsweg 87,
          D-44221 Dortmund, Germany}
\email{michael.voit@math.tu-dortmund.de, jeannette.woerner@math.tu-dortmund.de}

\begin{document}
\subjclass[2010]{Primary 34A05; Secondary 60F05, 60J60, 60B20, 82C22, 33C45 }
\keywords{Interacting particle systems, Calogero-Moser-Sutherland models, ODE in the freezing limit,
existence andd uniqueness of solutions, properties of solutions.}

\begin{abstract}
Multivariate Bessel processes   describe Calogero-Moser-Sutherland  particle models
 and are related with $\beta$-Hermite and $\beta$-Laguerre ensembles. 
They depend on a root system and a multiplicity $k$.
Recently, several limit theorems for $k\to\infty$ were derived where the limits depend on the solutions
of  associated ODEs in these freezing regimes.
In this paper we study the solutions of these ODEs which are are singular on the boundaries of their domains.
In particular we prove that for a start in arbitrary boundary points, the ODEs always admit unique solutions in their domains for $t>0$.
\end{abstract}

\date{\today}

\maketitle

\section{Introduction}
Calogero-Moser-Suther\-land  particle systems on  $\mathbb R$  with $N$ particles 
can be described as  multivariate Bessel processes on  closed Weyl chambers in  $\mathbb R^N$. These processes 
 are time-homogeneous diffusions with well-known generators of the transition semigroups,
and they are solution of the associated stochastic differential equations (SDEs); see  \cite{A, CGY, DV, R, R2, RV} for the background.
These  Bessel processes $(X_{t,k})_{t\ge0}$ are  classified via root systems,  a possibly multidimensional
multiplicity parameter (often called a coupling constant)  $k$ and by their starting points. Moreover, these processes are 
related to the $\beta$-Hermite and $\beta$-Laguerre ensembles of Dumitriu and Edelman \cite{DE1, DE2}; see e.g. \cite{AV1}.

We briefly recapitulate the generators of  $(X_{t,k})_{t\ge0}$ for the most important cases, the root systems $A_{N-1}$ and $B_N$.
For $A_{N-1}$, we have  $k\in]0,\infty[$, the processes live on the closed Weyl chamber
$$C_N^A:=\{x\in \mathbb R^N: \quad x_1\ge x_2\ge\ldots\ge x_N\},$$
and   the generators of the transition semigroups are
\begin{equation}\label{def-L-A} Lf:= \frac{1}{2} \Delta f +
 k \sum_{i=1}^N\Bigl( \sum_{j\ne i} \frac{1}{x_i-x_j}\Bigr) \frac{\partial}{\partial x_i}f ,
 \end{equation}
where we assume reflecting boundaries.
For $B_N$, we have $k=(k_1,k_2)\in]0,\infty[^2$, the processes live on 
$$C_N^B:=\{x\in \mathbb R^N: \quad x_1\ge x_2\ge\ldots\ge x_N\ge0\},$$
and the generators of the transition semigroups are
\begin{equation}\label{def-L-B} Lf:= \frac{1}{2} \Delta f +
 k_2 \sum_{i=1}^N \sum_{j\ne i} \Bigl( \frac{1}{x_i-x_j}+\frac{1}{x_i+x_j}  \Bigr)
 \frac{\partial}{\partial x_i}f 
\quad + k_1\sum_{i=1}^N\frac{1}{x_i}\frac{\partial}{\partial x_i}f \end{equation}
with  reflecting boundaries.
If we define the  weight functions $w_k$ of the form
\begin{equation}\label{def-wk}
w_k^A(x):= \prod_{i<j}(x_i-x_j)^{2k}, \quad\quad 
w_k^B(x):= \prod_{i<j}(x_i^2-x_j^2)^{2k_2}\cdot \prod_{i=1}^N x_i^{2k_1},\end{equation}
the generators may be written in a unified way as
\begin{equation}\label{def-L-general} Lf:= \frac{1}{2} \Delta f +
\frac{1}{2}  \nabla \ln w_k \cdot \nabla f.
 \end{equation}

The associated SDEs  then have the form
\begin{equation}\label{SDE-general}
dX_{t,k}= dB_t +  \frac{1}{2} \nabla(\ln w_k)(X_{t,k}) \> dt\end{equation}
on the interior of the associated Weyl chamber $C_N$ with an $N$-dimensional Brownian motion $(B_t)_{t\ge0}$.
It is known from  \cite{CGY, Sch, GM} that  for each starting point in the closed Weyl chamber $C_N$, 
 (\ref{SDE-general}) has a unique 
strong solution $(X_{t,k})_{t\ge0}$ whenever all  components of $k$ are positive. Moreover, 
if all components of $k$ are at least $ 1/2$, we have the following behaviour: If the process starts in the
 interior of $C_N$, then  $(X_{t,k})_{t\ge0}$  does not hit the boundary $\partial C_N$ of $C_N$  almost surely. If it starts at the boundary, it leaves the boundary immediately.  

Assume now that $k=\beta>0$ in the A-case and $k=(k_1,k_2)=(\nu\beta,\beta)$ with $\nu>0$ fixed and $\beta>0$
in the B-case. Then in both cases, the renormalized processes $(\tilde X_{t,k}:=\frac{1}{\beta}X_{t,k})_{t\ge0}$  satisfy
$d\tilde X_{t,k}= \frac{1}{\sqrt\beta}dB_t +  \frac{1}{2} \nabla(\ln w)(\tilde X_{t,k}) \> dt$
with $w$ of the form
\begin{equation}\label{def-w}
w^A(x):= \prod_{i<j}(x_i-x_j)^{2}, \quad\quad 
w_k^B(x):= \prod_{i<j}(x_i^2-x_j^2)^{2}\cdot \prod_{i=1}^N x_i^{2\nu}.\end{equation}
In \cite{AKM1, AKM2, AV1, V, VW}, several limit theorems like laws of large numbers and central limit theorems were derived
for $\tilde X_{t,k}$ and $X_{t,k}$ for $\beta\to\infty$, where the limits are described in terms
 of the solutions of the associated ordinary differential equations (ODEs) 
\begin{equation}\label{ODE-general}
\frac{dx(t)}{dt}=  \frac{1}{2} \nabla(\ln w)(x(t)) \end{equation}
in the interior of $C_N$. Unfortunately, only for  particular starting points $x_0$, the gradient systems 
(\ref{ODE-general}) can be solved explicitly. For instance, by \cite{AV1}, one has solutions of the form $x(t)=\sqrt{c+t}\cdot z$ for
 $c>0$ where the coordinates of $z\in\mathbb R^N$  are the ordered zeros of classical orthogonal polynomials of order $N$.
 Formally, these solutions also exist for $c=0$ where we start in the singular point $x(0)=0\in \partial C_N$.

The aim of this paper is to study several properties of the solutions of (\ref{ODE-general}) in a systematic way
for the most interesting root systems of types A,B, and D; see Section 4 for the details the the root systems $D_n$. 
In particular, we  show that all solutions of (\ref{ODE-general}) can be determined explicitely up to determining 
the zeros of polynomial of order $N$ via transformations of  (\ref{ODE-general}) using elementary symmetric polynomials. 
We also  obtain from this approach that for all starting points $x_0\in \partial C_N$ there is a unique continuous 
solution $x$ of (\ref{ODE-general}) with $x(t)$ in the interior of  $C_N$ for $t>0$.
 The arguments will be similar for  root systems of types A,B, and D; these three cases will be studied in the next three sections.

\section{The root system $A_{N-1}$} 

For the root system $A_{N-1}$, the ODE (\ref{ODE-general}) has the form
\begin{equation}\label{ODE-a}
\frac{dx(t)}{dt}= H(x(t)) \quad\text{with} \quad H(x):=
\Bigl( \sum_{j\ne1} \frac{1}{x_1-x_j},\ldots,\sum_{j\ne N} \frac{1}{x_N-x_j} \Bigr)
\end{equation}
on the interior $ W_N^A$ of  the Weyl chamber $ C_N^A$.
By classical results on ODEs the following is straight forward; see Lemma 2.1 of \cite{AV1}:

\begin{lemma}\label{deterministic-boundary-A}
For $\epsilon>0$ consider the open subset $U_\epsilon:=\{x\in C_N^A:\> d(x,\partial C_N^A)>\epsilon\}$
with the Euclidean distance $d$ on  $\mathbb R^N$. Then  $H$ is Lipschitz-continuous on each $U_\epsilon$, and
 for each  $x_0\in U_\epsilon$, the ODE (\ref{ODE-a}) with $x(0)=x_0$ admits a unique solution. 
This solution exists for all $t>0$  with $x(t)\in U_\epsilon$.
\end{lemma}

We next recapitulate from \cite{AV1} that (\ref{ODE-a}) has the following particular solutions:

\begin{lemma}\label{special-solution}
Let $(H_N)_{N\ge 0}$ be the Hermite polynomials which are orthogonal w.r.t.
 the density  $e^{-x^2}$ as in the monograph \cite{S}. Let $z:=(z_1,\ldots,z_N)\in  W_N^A$ be the vector consisting
of the ordered zeros of the Hermite polynomial $H_N$.
Then for each $c>0$, $x(t):=\sqrt{2t+c^2}\cdot {z}$ is the solution of  (\ref{ODE-a}) with $x(0)=cz$.
\end{lemma}

The growth of these particular solutions w.r.t.~the Euclidean norm on $\mathbb R^N$ is typical; see \cite{VW}:

\begin{lemma}\label{growth-a} For each solution $x$ of  (\ref{ODE-a}) with start in  $ W_N^A$,
\begin{equation}\label{ODE-norm}
\|x(t)\|^2 =N(N-1)t+\|x(0)\|^2.
\end{equation}
\end{lemma}

Lemma \ref{growth-a} motivates a decomposition of solutions into an easy radial part and and a spherical 
part  on
the unit sphere in $\mathbb R^N$. If doing so, the spherical
 parts of the special solutions in \ref{special-solution}
correspond to a stationary solution. This stationary solution
satisfies the following stability result:

\begin{lemma}\label{stabilily-lemma}
For each initial value $x_0\in W_N^A $, the solution $x$ of (\ref{ODE-a})  has the form
$$x(t)= \sqrt{ N(N-1)t+\|x_0\|^2}\cdot \phi(t)  \quad\quad(t\ge0)$$
where  $\phi$ satisfies
$$\|\phi(t)\|=1  \quad\quad\text{and}\quad\quad \lim_{t\to\infty} \phi(t) = \sqrt{\frac{2}{N(N-1)}}\cdot  z .$$
\end{lemma}

\begin{proof}
Using (\ref{ODE-norm}), we define
\begin{equation}\label{decomposition-phi} \phi(t) := 
(\phi_{1}(t), \ldots,  \phi_{N}(t)):= \frac{1}{ \sqrt{ N(N-1)t+\|x_0\|^2}}\cdot x(t)=
\frac{x(t)}{\|x(t)\|}
\end{equation}
with
$\| \phi(t) \|=1$. The ODE (\ref{ODE-a})  implies that
\begin{align}
\frac{d}{dt}(\phi_{i}(t)) &=  \frac{\dot x_i(t)}{\sqrt{ N(N-1)t+\|x_0\|^2}} - 
 \frac{N(N-1)\cdot x_i(t)}{2(N(N-1)t+\|x_0\|^2)^{3/2}}   \notag\\
 &=  \frac{1}{N(N-1)t+\|x_0\|^2} \Biggl(\sum_{j\ne i}  \frac{\sqrt{ N(N-1)t+\|x_0\|^2}}{x_i(t)-x_j(t)}-
 \frac{ N(N-1)}{2} \phi_{i}(t) \Biggr) \notag\\
&= \frac{1}{N(N-1)t+\|x_0\|^2} \Biggl(\sum_{j\ne i}  \frac{1}{\phi_{i}(t)-\phi_{j}(t)}
-\frac{ N(N-1)}{2} \phi_{i}(t) \Biggr).\notag
\end{align}
Therefore,
$$\psi(t) := \phi\Bigl( \frac{ N(N-1)}{2}t^2+ \|x_0\|^2t \Bigr) \quad(t\ge0)$$
satisfies
\begin{equation}\label{ODE-tilde}
\dot{\psi}_i(t) = \sum_{j\ne i}  \frac{1}{\psi_{i}(t)-\psi_{j}(t)}
-\frac{ N(N-1)}{2} \psi_{i}(t)\quad(i=1,\ldots,N)
\end{equation}
with $\psi(0)=\phi(0)=x_0/\|x_0\|$. The ODE (\ref{ODE-tilde}) is a gradient sytem
 $\dot{\psi}=(\nabla u)(\psi)$ with
$$u(y):= 2\sum_{i,j=1,\ldots,N, i< j}\ln(y_i-y_j) - \|y\|^2\cdot N(N-1)/4.$$
It now follows from Lemma 2.2 of \cite{AV1} (or see \cite{AKM1} or Section 6.7 of \cite{S})
that $u$ admits a unique local maximum on $C_N^A$,
 that this maximum is a global one, and that it located at 
$$\sqrt{\frac{2}{N(N-1)}}\cdot z$$
where, by (D.22) of \cite{AKM1}, this vector has $\|.\|_2$-norm 1.
 We conclude from Section 9.4 of \cite{HS} on gradient systems
that this point is an asymptotically stable equilibrium point of the ODE  (\ref{ODE-tilde}). 
This and (\ref{decomposition-phi}) now lead to the claim. 
\end{proof}

We cannot determine $x$ explicitly for arbitrary starting points and  $N$.
 On the other hand, Lemma \ref{growth-a} is a special case of the  observation that
for each symmetric polynomial $p$ in $N$ variables, $t\mapsto p(x(t))$ is a polynomial in $t$ which can be computed
explicitly. In this context we shall now employ 
 the elementary
 symmetric polynomials $e_k:=e_k^{N}$ ($k=0,\ldots,N$) in $N$ variables which are characterized by
\begin{equation}\label{symmetric-poly}
\prod_{k=1}^N (z-x_k) = \sum_{k=0}^{N}(-1)^{N-k}  e_{N-k}(x) z^k \quad\quad (z\in\mathbb C, \> x=(x_1,\ldots,x_N)).
\end{equation}
In particular,
$e_0=1, \> e_1(x)=\sum_{k=1}^N x_k ,\ldots, e_N(x)=\prod_{k=1}^N x_k$.
As each symmetric polynomial in $N$ variables is a polynomial in  $e_1,\ldots,e_N$ by a classical result,
 the following lemma shows
that all symmetric polynomials of
 $x$ are polynomials in $t$.

\begin{lemma}\label{symmetric-pol-pol-in-t}
For each initial value $x_0=(x_{0,1},\ldots, x_{0,N})\in W_N^A$, let $x$ be the solution  of (\ref{ODE-a}).
Then, for   $k=0 ,\ldots,N$, $t\mapsto e_k(x(t))$
 is a polynomial in $t$ of degree (at most) $\lfloor\frac{k}{2}\rfloor $ where the leading coefficient of
order $\lfloor\frac{k}{2}\rfloor $ is given by
$$ \frac{(-1)^l\cdot N!}{2^l \cdot l! (N-2l)!} \quad (k=2l\le N) \quad\text{and}\quad
  \frac{(-1)^l \cdot (N-1)!}{2^l \cdot l! (N-2l-1)!}\cdot\sum_{j=1}^N x_{0,j} \quad (k=2l+1\le N).$$
\end{lemma}

\begin{proof}
This is trivial for $k=0$ and can be easily verified for $k=1$.
For $k\ge2$ we use induction on $k$ and use the following notion:
For a non-empty set $S\subset \{1,\ldots,N\}$, the vector $x\in\mathbb R^{|S|}$ is the vector with the 
coordinates $x_i$ for $i\in S$ in the natural ordering on $S$. We then have for $k\ge2$ that
$$\frac{d}{dt} e_k(x(t))= \sum_{j=1}^N \frac{dx_j(t)}{dt}
 \cdot e_{k-1}^{N-1}(x_{\{1,\ldots,N\}\setminus\{j\}}(t)).$$
Therefore, by the differential equation,
\begin{align}\label{elementary-symm-a-1}
\frac{d}{dt} e_k(x(t))&=  \sum_{j=1}^N \sum_{i: i\ne j} 
\frac{  e_{k-1}^{N-1}(x_{\{1,\ldots,N\}\setminus\{j\}}(t))}{x_j(t)-x_i(t)}\\
&= \frac{1}{2}     \sum_{i,j=1,\ldots,n; i\ne j}
\frac{  e_{k-1}^{N-1}(x_{\{1,\ldots,N\}\setminus\{j\}}(t))- e_{k-1}^{N-1}(x_{\{1,\ldots,N\}\setminus\{i\}}(t))  }{x_j(t)-x_i(t)}.\notag
\end{align}
Moreover, simple combinatorial computations yield for $i\neq j$
\begin{eqnarray}\label{elementary-symm-a-2}
\lefteqn {e_{k-1}^{N-1}(x_{\{1,\ldots,N\}\setminus\{j\}}(t))- e_{k-1}^{N-1}(x_{\{1,\ldots,N\}\setminus\{i\}}(t))=}\\
&&\quad\quad\quad\quad\quad\quad\quad\quad\quad\quad\quad\quad(x_i(t)-x_j(t))e_{k-2}^{N-2}(x_{\{1,\ldots,N\}\setminus\{i,j\}}(t))\nonumber
\end{eqnarray}
and 
\begin{equation}\label{elementary-symm-a-3}
 \sum_{i,j=1,\ldots,N; i\ne j}e_{k-2}^{N-2}(x_{\{1,\ldots,N\}\setminus\{i,j\}}(t))= (N-k+2)(N-k+1)e_{k-2}^{N}(x(t)).
\end{equation}
Therefore, by (\ref{elementary-symm-a-1})-(\ref{elementary-symm-a-3}),
\begin{equation}\label{elementary-symm-a-4}
\frac{d}{dt} e_k(x(t))= - \frac{1}{2} (N-k+2)(N-k+1)e_{k-2}^{N}(x(t)).
\end{equation}
This recurrence relation and the cases $k=0,1$ now  lead to the claim.
\end{proof}

\begin{remark}\label{centering-wlog}
For each $r\in\mathbb R$
 and each solution $x$ of (\ref{ODE-a}), the function $x^r(t):= x(t)+r\cdot(1,\ldots,1)$
is also a  solution  of (\ref{ODE-a}).
This implies that we may assume $\sum_{j=1}^N x_{0,j} =0$ without loss of generality for our initial conditions.
If we do so, the degrees of the polynomials $t\mapsto e_k(x(t))$ for odd $k$ can be chosen to be even smaller.
\end{remark}

Lemma \ref{symmetric-pol-pol-in-t} and Eq.~(\ref{symmetric-poly})
 can be used to compute all solutions of (\ref{ODE-a}) on $W_N^A$.
 First, one has to determine the polynomials
 $e_k(x(t))$ ($k=1,\ldots,N)$. In a second step, one has to determine the ordered, different 
zeros of the polynomials 
in (\ref{symmetric-poly}) 
from the coefficients of the polynomials. This relation  corresponds to some  diffeomorphism.
We present an example.

\begin{example}\label{example-a} Let $N=3$. Let $x_0\in C_3^{A}$ with $x_{0,1}+x_{0,2}+x_{0,3}=0$ according to Remark
\ref{centering-wlog}. We here obtain
$$e_1(x(t))=0, \quad e_1(x(t))=-3t + e_2(x_0),  \quad e_3(x(t))=e_3(x_0)$$
and thus 
\begin{equation}\label{cardano-case}
\prod_{k=1}^3 (z-x_k(t)) = z^3 + (e_2(x_0)-3t)z - e_3(x_0).\end{equation}
We now apply Cardano's formula in the casus irreducibilis.
We  first observe that the existence of 3 real zeros implies  $3t-e_2(x_0)>0$; we have the 
solutions
$$ x_k(t)= \sqrt{ 4t-\frac{4}{3}e_2(x_0)} 
\cdot \cos\Bigl(  \frac{1}{3} \arccos \Bigl(\frac{\sqrt{27}}{2} \frac{e_3(x_0)}{(3t-e_2(x_0))^{3/2} } \Bigr) +
 \frac{2}{3}(1-k)\pi\Bigr)$$
for $k=1,2,3$. The correct ordering $x_1(t)\ge x_2(t)\ge x_3(t)$ 
here follows easily from the case $t\to\infty$ in which case 
\begin{align}
 x(t)&= \sqrt{ 4t-\frac{4}{3}e_2(x_0)} \cdot\Bigl( (\sqrt{3/4}, 0,-\sqrt{3/4})+ O(t^{-3/2})\Bigr)\notag\\
&= \sqrt{ 6t-2e_2(x_0)} \cdot (\sqrt{1/2}, 0,-\sqrt{1/2}) + O(t^{-1})  \quad=:\quad \tilde x(t) + O(t^{-1})
\notag
\end{align}
with a solution $\tilde x$ of our differential equation of the type of Lemma \ref{special-solution}.
This improves Lemma \ref{stabilily-lemma} in a quantitative way for $N=3$.

We also remark that the discriminant of the polynomial (\ref{cardano-case})  is 
$$\Delta:= e_3(x_0)^2/4 + (e_2(x_0)-3t)^3/27.$$
By  Cardano's formula, $\Delta=0$ holds if and only if we have multiple (real)
 zeros in (\ref{cardano-case}), i.e., a point in $\partial C_3^{A}$. Hence,
 if we formally start our solution at some $x_0\in \partial C_3^{A}$ with $x_{0,1}+x_{0,2}+x_{0,3}=0$,
then  $x(t)\in W_3^{A}$ exists for all 
 all $t>0$.
By Remark \ref{centering-wlog}, this observation also holds 
for arbitrary starting points in $ C_3^{A}$.
\end{example}

The existence of solutions with a start on the  boundary in Example \ref{example-a} holds for arbitrary dimensions:

\begin{theorem}\label{ode-ex-unique-a-thm}
Let $N\ge 2$. 
Then, for each starting point $x_0\in C_N^A$, the ODE (\ref{ODE-a})
has a unique solution for all $t\ge0$ in the following sense:
For each $x_0\in C_N^A$ there is a unique continuous function
$x:[0,\infty[\to C_N^A$ with $x(0)=x_0$ such that for $t\in]0,\infty[$, $x(t)\in  W_N^{A} $
holds, and $x:]0,\infty[\to  W_N^{A} $  satisfies (\ref{ODE-a}).
\end{theorem}

\begin{proof}
By Lemma \ref{deterministic-boundary-A}, it suffices to
assume $x_0\in\partial C_N^A$.  We  transform (\ref{ODE-a})
via elementary
symmetric polynomials as in Lemma \ref{symmetric-pol-pol-in-t}.
 Then, for $t>0$ and $k=2,\ldots,N$,
\begin{equation}\label{recursive-symm-eq-a}
\frac{d}{dt}\Bigl( e_k^{N}(x(t))\Bigr)=-\frac{1}{2} (N-k+2)(N-k+1)  e_{k-2}^{N}(x(t))
\end{equation}
where in particular,
$$\frac{d}{dt}\Bigl( e_2^{N}(x(t))\Bigr)=-\frac{1}{2} N(N-1).$$
Moreover, for $k=1$,
\begin{equation}\label{symm-eq-a-1}
\frac{d}{dt}\Bigl( e_1^{N}(x(t))\Bigr)= \sum_{i,j=1,\ldots,N; i\ne j} \frac{ 1}{x_i(t)-x_j(t)}=0.
\end{equation}
 These equations form an ODE whose solutions $\tilde e(t):=( e_1^{N}(x(t)),\ldots, e_N^{N}(x(t)))$ 
are polynomials in $t$ 
 by Lemma \ref{symmetric-pol-pol-in-t}.
As the mapping
\begin{align}\label{special-mapping}
e: C_N^A&\to E_N:=\Bigl\{y=(y_1,\ldots,y_N)\in\mathbb R^N: \quad\text{the polynomial}\notag\\
 &\quad \quad\quad\quad P(z):=\sum_{j=0}^{N}(-1)^{N-j} y_{N-j} z^j
  \text{ with}\>\> y_0=1 \>\>\text{has has only real zeros}\Bigr\}, \notag\\
 x &\mapsto (e_1(x),\ldots,e_N(x)) 
\end{align}
is continuous and injective,  the ODE (\ref{ODE-a}) with start in $x_0$
has at most one solution.

For the existence of a solution we notice that the restriction
$e: W_N^A \to e( W_N^A) $
of the mapping $e$ is a diffeomorphism while the extension to the closure
$$e: C_N^A\to e( C_N^A)= \overline{e( W_N^A)} \subset E_N $$
 is  a homeomorphism.
We claim that the inverse mapping of $e$ transforms 
the existent polynomial solutions of the ODEs (\ref{symm-eq-a-1}) and (\ref{recursive-symm-eq-a})
 back into solutions of the ODE
in  (\ref{ODE-a}) in the sense of the theorem. For this we prove that
for any starting point $x_0\in\partial C_N^A$ in (\ref{ODE-a}) and its image
$e(x_0)\in \overline{e( W_N^A)}  $,
 the  solution $\tilde e(t)$ ($t\ge0)$ of the ODEs
 (\ref{symm-eq-a-1}) and (\ref{recursive-symm-eq-a}) with $\tilde e(0)=e(x_0)$ satisfies
$\tilde e(t)\in e(W_N^A) $  for  $t>0$ sufficiently small.
If this is shown it follows that the preimage of  $(\tilde e(t))_{t\ge0}$ under $e$ is a solution of 
  (\ref{ODE-a}) in the sense of the theorem.

To prove this statement, we 
 recapitulate that for each starting point in $e(W_N^A)$
the  solution of the ODEs (\ref{symm-eq-a-1}) and (\ref{recursive-symm-eq-a}) satisfies 
 $\tilde e(t)\in e( W_N^A)$ for $t\ge0$, and that   for all fixed $t\ge0$,
the solutions $\tilde e(t)$ depend continuously from arbitrary starting points in $\mathbb R^N$
by a classical theorem on ODEs.
We thus conclude that for each starting point in $\tilde e(0)\in\overline{e( W_N^A)}$ we have
 $\tilde e(t)\in\overline{e( W_N^A)}$ for $t\ge0$.

We next observe that for each solution $x$ of    (\ref{ODE-a}) with start in 
$x(0)\in \partial C_N^A$, and for each $c>0$, the function $t\mapsto \frac{1}{\sqrt c} x(ct)$  also  solves
 the ODE in   (\ref{ODE-a}) with  start in $x(0)/\sqrt c$. Moreover, for each solution
 $(\tilde e(t))_{t\ge0}$ of the ODEs  (\ref{symm-eq-a-1}) and (\ref{recursive-symm-eq-a}) with start in
$e(x(0))$, and each $c>0$, the function 
\begin{equation}\label{modified-solution-c-a}
t\mapsto \tilde e_c(t):= ( c^{-1/2} \tilde e_1(ct),c^{-2/2} \tilde e_2(ct),\ldots, c^{-N/2} \tilde e_N(ct) )
\end{equation}
 is also a solution of 
 the ODEs  (\ref{symm-eq-a-1}) and (\ref{recursive-symm-eq-a}) with start in
$$ ( c^{-1/2}e_1(x(0)),\ldots, c^{-N/2}e_N(x(0)) ).$$

Assume now that
there is a starting point $x(0)\in \partial C_N^A$ and some $t_0>0$, such that the solution  $(e(t))_{t\ge0}$
 of (\ref{symm-eq-a-1}) and (\ref{recursive-symm-eq-a}) with start in $e(x(0))$ satisfies
$\tilde e(t)\not\in e( W_N^A)$ for $t\in[0,t_0]$.
This means that for $t\in[0,t_0]$, 
 \begin{equation}\label{solution-negative}
\tilde e(t)\in \overline{e( W_N^A)}\setminus e( W_N^A )
\subset Y:=\{ y\in\mathbb  R^N:\> \tilde D(y)=0\}
\end{equation}
for the discriminant mapping $\tilde D:\mathbb  R^N\to\mathbb  R$  which is defined as follows:
For $x=(x_1,\ldots,x_N)\in\mathbb  R^N$ we form the discriminant
$$D(x):=\prod_{1\le i< j\le N} (x_i-x_j)^2.$$
$D$ is  a symmetric polynomial in $x_1,\ldots,x_N$ and thus,
by a classical result on elementary symmetric polynomials, a polynomial $\tilde D$ in $e_1(x),\ldots,e_N(x)$.
This is the  discriminant mapping used above.
In summary we obtain that $\tilde D(\tilde e(t))=0$ for  $t\in[0,t_0]$ with some polynomial $\tilde D\circ \tilde e$
which yields that  $\tilde D(\tilde e(t))=0$ for all  $t\ge0$. As 
$Y\cup  e(W_N^A )=\emptyset$, we conclude that 
$\tilde e(t)\not\in e(W_N^A )$ holds for all $t\ge0$.
Therefore,  for all $c>0$, the modified solutions from (\ref{modified-solution-c-a}) satisfy
$ \tilde e_c(t)\not\in  e( W_N^A)$ for $t\ge0$.
For $c\to\infty$, the starting points $ \tilde e_c(0)$ of these solutions tend to the vector $0\in\mathbb R^N$.

On the other hand, if $z_1>z_2>\ldots> z_N$ are the ordered zeros of the Hermite polynomial $H_N$, then 
$x_0(t):=\sqrt{2t}\cdot(z_1,\ldots,z_N)$ solves (\ref{ODE-a}) 
with start in $0$ in the sense of our theorem by \cite{AV1}. This implies that 
$(\tilde e_0(t))_{t\ge0}$ is the solution of  (\ref{symm-eq-a-1}) and (\ref{recursive-symm-eq-a}) with start in $0$
where for this solution $\tilde e_0(t) \in e(W_N^A)$ holds for all $t>0$.
We thus obtain a contradiction to the continuous dependence of the solutions of the ODEs
 (\ref{symm-eq-a-1}) and (\ref{recursive-symm-eq-a}) from the starting points.
We thus obtain that there is no  starting point $x(0)\in \partial C_N^A$ and no $t_0>0$ with
$\tilde e(t)\not\in e( W_N^A)$ for $t\in[0,t_0]$. This completes the proof.
\end{proof}

Theorem \ref{ode-ex-unique-a-thm} can be supplemented by the following observation:

\begin{lemma}\label{lower-boundary-a}
Let $(x(t))_{t\ge0}\subset W_N^A$ be a solution of the ODE (\ref{ODE-a}). Then there is a unique $t_0<0$ such that
 $(x(t))_{t\ge t_0}$ is a solution of (\ref{ODE-a}) with $x(t_0)\in\partial C_N^A$ in the sense of
 Theorem \ref{ode-ex-unique-a-thm}, i.e.,    $(x(t))_{t> t_0}\subset W_N^A$ solves  (\ref{ODE-a}).
\end{lemma}

\begin{proof} 
By Lemma \ref{deterministic-boundary-A}, the RHS of  (\ref{ODE-a}) is locally Lipschitz continuous on  $W_N^A$.
We thus have a maximal solution  $(x(t))_{t> t_0}\subset W_N^A$ with some $t_0\in[-\infty,0[$.
On the other hand, by Lemma \ref{growth-a}, this solution must satisfy
$\|x(t)\|^2 =N(N-1)t+\|x(0)\|^2$ for $t>t_0$ which implies that $t_0\ge -\|x(0)\|^2/(N(N-1))$ and that $x(t)$
 remains in some compactum for $t\downarrow t_0$. Moreover, we have  $d(x(t), \partial C_N^A)\to0$ for  $t\downarrow t_0$ 
by a classical result on ODEs.
We now consider the  homeomorphism
$$e: C_N^A\to e( C_N^A)= \overline{e( W_N^A)} \subset E_N $$
from the preceding proof. As by the results of the preceding proof
 $\lim_{t\downarrow t_0} e(x(t))$ exists, we conclude that
 $\lim_{t\downarrow t_0}  x(t)=:x(t_0)\in \partial C_N^A$ exists. This yields the claim.
\end{proof}

\section{The root system  $B_{N}$}

For the root systems $B_N$ we fix some constant $\nu>0$ as described in the introduction.
 The ODE (\ref{ODE-general}) here  has the form
\begin{equation}\label{ODE-b}
\frac{dx(t)}{dt}= H_\nu(x(t)) \quad\text{with}
 \quad H_\nu(x):=\left(\begin{matrix}\sum_{j\ne1}\Bigl( \frac{1}{x_1-x_j}+  
 \frac{1}{x_1+x_j}\Bigr)+\frac{\nu}{x_1}\\ \vdots\\ \sum_{j\ne N}
 \Bigl(\frac{1}{x_N-x_j}+   \frac{1}{x_N+x_j}\Bigr)+\frac{\nu}{x_N}\end{matrix}\right)
\end{equation}
on the interior $ W_N^B$ of  the closed Weyl chamber $ C_N^B$. We  recapitulate from \cite{AV1}:

\begin{lemma}\label{deterministic-boundary-B}
Let $\nu>0$.  For $\epsilon>0$ consider the open subset
 $$U_\epsilon:=\{x\in C_N^B:\> x_N>\frac{\epsilon \nu}{N-1}, \quad\text{and}\quad x_i-x_{i+1}>\epsilon
 \quad\text{for}\quad i=1,\ldots,N-1  \}$$
of $C_N^B$. 
 Then for each  $x_0\in U_\epsilon$, the ODE (\ref{ODE-b}) with $x(0)=x_0$ admits a unique solution. 
This solution exists for all $t>0$  with $x(t)\in U_\epsilon$.
\end{lemma}

We next recapitulate from \cite{AV1} that (\ref{ODE-b}) has some particular solutions. 
For this we recall that for $\alpha>0$
the Laguerre polynomials  $(L_n^{(\alpha)})_{n\ge0}$ are orthogonal w.r.t.~the density $e^{-x}\cdot x^{\alpha}$; see the monograph \cite{S} for details.  We know from \cite{AV1}:

\begin{lemma}\label{special-solution-B1}
Let $\nu>0$. Let $z_1^{(\nu-1)}>\ldots>z_N^{(\nu-1)}>0$ be the ordered zeros of $L_N^{(\nu-1)}$, and
$y:=(y_1, \ldots, y_N)\in W_N^B$ with
\begin{equation}\label{y-max-B1}
2(z_1^{(\nu-1)},\ldots, z_N^{(\nu-1)})= (y_1^2, \ldots, y_N^2).
\end{equation}
 Then for each $c>0$, $x(t):=\sqrt{t+c^2}\cdot y $ is a solution of (\ref{ODE-b}). 
\end{lemma}

Again, the growth of these particular solutions is typical; see \cite{VW}:

\begin{lemma}\label{growth-b} For each solution $x$ of  (\ref{ODE-b}) with start in  $ W_N^B$,
\begin{equation}\label{ODE-norm-b}
\|x(t)\|^2 =2N(N+\nu-1)t+\|x(0)\|^2.
\end{equation}
\end{lemma}

We again decompose solutions of   (\ref{ODE-b})
 into an easy radial part and and a spherical 
part where  the spherical parts of the  solutions in \ref{special-solution}
correspond to a stationary solution. This stationary solution
satisfies the following stability result:

\begin{lemma}\label{stabilily-lemma-b}
For each starting point $x_0\in W_N^B$, the solution $x$ of (\ref{ODE-b}) has the form
$$x(t)= \sqrt{ 2N(N+\nu-1)t+\|x_0\|^2}\cdot \phi(t)  \quad\quad(t\ge0)$$
with
$$\|\phi(t)\|=1 \quad\quad\text{and}\quad\quad \lim_{t\to\infty} \phi(t) = \frac{2}{N(N-1)} y $$
and the vector $y$ from Lemma \ref{special-solution-B1}.
\end{lemma}

\begin{proof} The proof is similar to that of Lemma \ref{stabilily-lemma}.
Using (\ref{ODE-norm-b}), we define
\begin{equation}\label{decomposition-phi-b} \phi_(t) := 
 \frac{1}{ \sqrt{ 2N(N+\nu-1)t+\|x_0\|^2}}\cdot x(t)=
\frac{x(t)}{\|x(t)\|}
\end{equation}
with
$\| \phi(t) \|=1$. (\ref{ODE-b})  implies that
\begin{align}
\frac{d}{dt}(\phi_{i}(t)) &=  \frac{\dot x_i(t)}{\sqrt{ 2N(N+\nu-1)t+\|x_0\|^2}} - 
 \frac{N(N+\nu-1)\cdot x_i(t)}{(2N(N+\nu-1)t+\|x_0\|^2)^{3/2}}   \notag\\
 &=  \frac{1}{2N(N+\nu-1)t+\|x_0\|^2} \Biggl(\sum_{j\ne i}  \frac{1}{\phi_{i}(t)-\phi_{j}(t)}
+\notag\\
&+\sum_{j\ne i}  \frac{1}{\phi_{i}(t)+\phi_{j}(t)}
+   \frac{ \nu }{\phi_{i}(t)}
- N(N+\nu-1) \phi_{i}(t) \Biggr).\notag
\end{align}
Hence,
$\psi(t) := \phi\Bigl( N(N+\nu-1)t^2+ \|x_0\|^2t \Bigr)$ for $t\ge0$
satisfies
\begin{align}\label{ODE-tilde-b}
\dot{\psi}_i(t) = &\sum_{j\ne i}  \frac{1}{\psi_{i}(t)-\psi_{j}(t)}+
\sum_{j\ne i}  \frac{1}{\psi_{i}(t)+\psi_{j}(t)}
 \notag\\
&\quad
+\frac{ \nu}{  \psi_{i}(t)} -   N(N+\nu-1)\cdot \psi_{i}(t)
\end{align}
for $i=1,\ldots,N$ with $\psi(0)=\phi_0(0)=x_0/\|x_0\|$. 
The ODE (\ref{ODE-tilde-b}) is a gradient system
 $\dot{\psi}=(\nabla u)(\psi)$ with
$$u(y):= \sum_{i,j=1,\ldots,N, i< j}(\ln(y_i-y_j) +\ln(y_i+y_j)) +\nu\sum_{i=1}^N \ln y_i  - \frac{N(N+\nu-1)}{2} 
\|y\|^2.$$
It now follows from Lemma 3.2 of \cite{AV1} (or see \cite{AKM2} or Section 6.7 of \cite{S})
that  $u$ admits a unique local maximum on $C_N^B$,
 that this maximum is a global one, and that it located at 
$y$ with 
$$(y_1^2,\ldots,y_N^2)=\frac{1}{N(N+\nu-1)}(z_1^{(\nu-1)},\ldots, z_N^{(\nu-1)}).$$
 We notice that  $\|y\|=1$ holds; this follows either from Lemmas \ref{special-solution-B1} and \ref{growth-b} or from
 (C.10) in \cite{AKM2}.
 These observations and  (\ref{decomposition-phi-b}) now lead to the claim as in the proof of Lemma
\ref{stabilily-lemma}.
\end{proof}

In order to describe the general solutions $x(t)$ of (\ref{ODE-b}) 
 we again  use 
 the elementary
 symmetric polynomials $e_k^{N}$ in $N$ variables  and put
$$\tilde e_k(x):=\tilde e_k^{N}(x):= e_k^{N}(x_1^2,\ldots, x_N^2)  \quad\quad (k=0,\ldots,N).$$

\begin{lemma}\label{symmetric-pol-pol-in-t-b}
For each  $x_0\in W_N^B$, consider the solution $x(t)$ of  (\ref{ODE-b}).
 Then, for   $k=0 ,\ldots,N$, $t\mapsto\tilde e_k(x(t))$
 is a polynomial in $t$ of degree  ${k} $ with leading coefficient 
$$  2^k (N+\nu-1)(N+\nu-2)\cdots(N+\nu-k)\cdot {N\choose k}    \quad (k\le N) .$$
\end{lemma}

\begin{proof}
The statement is trivial for $k=0$ and follows from (\ref{ODE-norm-b}) for $k=1$.
 For $k\ge2$ we use induction on $k$. We  use the notations  $x_S(t)$ from Section 2 and
 put $\tilde e_k^R(x):= e_k^R(x_1^2,\ldots, x_R^2)$ for $R=1,\ldots,N$.
 Then for $k\ge2$,
$$\frac{d}{dt}\tilde e_k(x(t))=2 \sum_{j=1}^N \frac{dx_j(t)}{dt} \cdot x_j(t)
 \cdot\tilde e_{k-1}^{N-1}(x_{\{1,\ldots,N\}\setminus\{j\}}(t)).$$
Therefore, by (\ref{ODE-b}),
\begin{align}\label{elementary-symm-b-1}
&\frac{d}{dt}\tilde e_k(x(t))= 2
 \sum_{j=1}^N \Biggl(2\sum_{i: i\ne j}\frac{x_j(t)^2}{x_j(t)^2-x_i(t)^2} \> +\> \nu\Biggr)
 \tilde e_{k-1}^{N-1}(x_{\{1,\ldots,N\}\setminus\{j\}}(t))\notag\\
&= 2    \sum_{i,j=1,\ldots,N; i\ne j}
\frac{x_j(t)^2 \tilde e_{k-1}^{N-1}(x_{\{1,\ldots,N\}\setminus\{j\}}(t))-x_i(t)^2\tilde e_{k-1}^{N-1}(x_{\{1,\ldots,N\}\setminus\{i\}}(t))  }{x_j(t)^2-x_i(t)^2}\notag \\
&\quad\quad +2\nu  \sum_{j=1}^N \tilde e_{k-1}^{N-1}(x_{\{1,\ldots,N\}\setminus\{j\}}(t))
.
\end{align}
Simple combinatorial computations show that for  $k\le N-1$,
\begin{eqnarray}\label{elementary-symm-b-2}
\lefteqn{x_j(t)^2 \tilde  e_{k-1}^{N-1}(x_{\{1,\ldots,N\}\setminus\{j\}}(t))-x_i(t)^2\tilde e_{k-1}^{N-1}(x_{\{1,\ldots,N\}\setminus\{i\}}(t))=}\\
 &&\quad\quad\quad\quad\quad\quad\quad\quad\quad\quad(x_j(t)^2-x_i(t)^2) \tilde e_{k-1}^{N-2}(x_{\{1,\ldots,N\}\setminus\{i,j\}}(t))\nonumber
\end{eqnarray}
and
\begin{equation}\label{elementary-symm-b-4}
 \sum_{i,j=1,\ldots,N; i\ne j}\tilde e_{k-1}^{N-2}(x_{\{1,\ldots,N\}\setminus\{i,j\}}(t))=
 (N-k+1)(N-k) \tilde e_{k-1}^{N}(x(t)).
\end{equation}
Moreover,
\begin{equation}\label{elementary-symm-b-3}
x_j(t)^2 \tilde  e_{N-1}^{N-1}(x_{\{1,\ldots,N\}\setminus\{j\}}(t))-x_i^2\tilde e_{N-1}^{N-1}(x_{\{1,\ldots,N\}\setminus\{i\}}(t))=0.
\end{equation}
Furthermore,
\begin{equation}\label{elementary-symm-b-5}
 \sum_{j=1}^N \tilde e_{k-1}^{N-1}(x_{\{1,\ldots,N\}\setminus\{j\}}(t))=(N-k+1)\tilde  e_{k-1}^{N}(x(t)).
\end{equation}
Therefore, by (\ref{elementary-symm-b-1})-(\ref{elementary-symm-b-5}), for $k\le N$,
\begin{equation}\label{elementary-symm-b-6}
\frac{d}{dt}\tilde e_k(x(t))=  2 (N-k+1)(N-k+\nu)\tilde e_{k-1}^{N}(x(t)).
\end{equation}
This recurrence relation and the known cases $k=0,1$ lead easily to the claim.
\end{proof}

\begin{example} Let $\nu>0$ and $N=2$. Assume that we start in $x_0=(x_{0,1},x_{0,2})\in C_2^{B}$. Then, 
 (\ref{elementary-symm-b-6}) and Lemma \ref{growth-b}  imply that
$$(z-x_1(t)^2)(z-x_2(t)^2)= z^2 -(x_1(t)^2+x_2(t)^2)z+x_1(t)^2x_2(t)^2$$
with 
$$x_1(t)^2+x_2(t)^2=4(1+\nu)t+\|x_0\|^2; \quad
x_1(t)^2x_2(t)^2= 4\nu(1+\nu)t^2+ 2\nu\|x_0\|^2t+ x_{0,1}^2x_{0,2}^2.$$
  Since the components of $x$ are non-negative, this yields that
\begin{align}
x_1(t)=&\left(\frac{1}{2}\Biggl(4(1+\nu)t+\|x_0\|^2+ 
\sqrt{ 16(1+\nu)t^2+8\|x_0|^2t+(x_{0,1}^2-x_{0,2}^2)^2}\Biggr)\right)^{1/2},\notag\\
x_2(t)=&\left(\frac{1}{2}\Biggl(4(1+\nu)t+\|x_0\|^2-
 \sqrt{ 16(1+\nu)t^2+8\|x_0\|^2t+(x_{0,1}^2-x_{0,2}^2)^2}\Biggr)\right)^{1/2}\notag
\end{align}
This implies in particular that for $t\to\infty$, 
\begin{align}
x_1(t)=&\left(\frac{1}{2}(4(1+\nu)t+\|x_0\|^2)\right)^{1/2}\left(1+\sqrt{\frac{1}{1+\nu}+
\frac{\frac{\nu}{1+\nu}\|x_0\|^4-4x_{1,0}^2x_{2,0}^2}{(4(1+\nu)t+\|x_0\|^2)^2}}\right)^{1/2}     \notag\\
=&\left(\frac{1}{2}(4(1+\nu)t+\|x\|^2)\right)^{1/2}\left(1+\sqrt{\frac{1}{1+\nu}}\right)^{1/2}+O(t^{-1/2}) \notag
\end{align}
and in the same way
\begin{equation}
x_2(t)=\left(\frac{1}{2}(4(1+\nu)t+\|x_0\|^2)\right)^{1/2}
\left(1-\sqrt{\frac{1}{1+\nu}}\right)^{1/2}+O(t^{-1/2}), 
\end{equation}
which may be seen as a quantitative version of Lemma \ref{stabilily-lemma-b} for $N=2$.
We also observe that our solutions  $x$  exist when we start
 at any point $x_0\in\partial C_2^{B}$ 
and that for these solutions,  $x(t)\in W_2^{B}$ holds for all $t>0$.
\end{example}

This last observation holds for all $N\ge2$:

\begin{theorem}\label{ode-ex-unique-b-thm}
Let $N\ge 2$ and $\nu>0$. 
Then, for each starting point $x_0\in C_N^B$, the ODE  (\ref{ODE-b})
has a unique solution for all $t\ge0$ in the sense of Theorem \ref{ode-ex-unique-a-thm}.
\end{theorem}

\begin{proof} 
The proof is similar to that of Theorem \ref{ode-ex-unique-a-thm}.
  We  keep the notations from there
and  describe the modifications.
 We again transform solutions $x$ of
the ODE  (\ref{ODE-b}). 
 Using the notation $\tilde e_k^R(x):= e_k^R(x_1^2,\ldots, x_R^2)$  as above, we
conclude from the proof of Lemma \ref{symmetric-pol-pol-in-t-b} that
for $t>0$ and $k=2,\ldots,N$, 
\begin{equation}\label{elementary-symm-b-1-k}
\frac{d}{dt}\Bigl(\tilde e_k^N(x(t))\Bigr)=2 (N-k+1)(N-k+\nu)\tilde e_{k-1}^{N}(x(t)).
\end{equation}
Moreover, for $k=1$,
\begin{equation}\label{symm-eq-b-1}
\frac{d}{dt}\Bigl(\tilde e_1^{N}(x(t))\Bigr)= 2N(N+\nu-1).
\end{equation}
 (\ref{symm-eq-b-1}) and (\ref{elementary-symm-b-1-k}) form an ODE, whose
 solutions $\tilde e(t):=(\tilde e_1^{N}(x(t)),\ldots, \tilde e_N^{N}(x(t)))$ 
are polynomials in $t$ where for $k=1,\ldots,N$, the $k$-th component  $\tilde e_k(t)$ 
is a polynomial  of maximal order $k$.
As the mapping $e: C_N^B\to E_N $ with $E_N$ as in  (\ref{special-mapping}) and 
$ e(x):=(\tilde e_1^{N}(x),\ldots, \tilde e_N^{N}(x)$ 
is continuous and injective,  (\ref{ODE-b})
has at most one solution.

For the existence of a solution we proceed as in the proof of  Theorem \ref{ode-ex-unique-a-thm}
and show that
for any starting point $x_0\in\partial C_N^B$  and its image
$\tilde e(x_0)\in \overline{e( W_N^B)} $,
 the  solution $\tilde e(t)$ ($t\ge0)$ of the ODEs
 (\ref{symm-eq-b-1}) and (\ref{elementary-symm-b-1-k}) with $\tilde e(0)=e(x_0)$ satisfies
$\tilde e(t)\in e(  W_N^B) $  for  $t>0$ sufficiently small.
To prove this, we first observe that by the same reasons
 as in the proof of  Theorem \ref{ode-ex-unique-a-thm},
 for each starting point in $\tilde e(0)\in\overline{e(W_N^B)}$ we have
 $\tilde e(t)\in\overline{e(W_N^B) }$ for all $t\ge0$.
Moreover, for each prospective solution $(x(t))_{t\ge0}$ of  (\ref{ODE-b}) with start in 
$x(0)\in \partial C_N^A$, and for each $c>0$, the function $t\mapsto \frac{1}{\sqrt c} x(ct)$  also  is also a 
 prospective solution of   (\ref{ODE-b}) with  start in $x(0)/\sqrt c$. This observation corresponds with the fact
that for each solution
 $(\tilde e(t))_{t\ge0}$ of the ODEs  (\ref{symm-eq-b-1}) and (\ref{elementary-symm-b-1-k}) with start in
$e(x(0))$, and each $c>0$, the function 
\begin{equation}\label{modified-solution-c-b}
t\mapsto  ( c^{-1} \tilde e_1(ct),c^{-2} \tilde e_2(ct),\ldots, c^{-N} \tilde e_N(ct) )
\end{equation}
 is also a solution of 
 (\ref{symm-eq-b-1}) and (\ref{elementary-symm-b-1-k}) with start in
$( c^{-1}\tilde e_1(x(0)),\ldots, c^{-N}\tilde  e_N(x(0)) ).$

Assume now that
there is a starting point $x(0)\in \partial C_N^B$ and some $t_0>0$, such that the solution  $(\tilde e(t))_{t\ge0}$
 of (\ref{symm-eq-b-1}) and (\ref{elementary-symm-b-1-k}) with start in $\tilde e(x(0))$ satisfies
$\tilde e(t)\not\in e(W_N^B )$ for $t\in[0,t_0]$.
We notice that $\partial C_N^B$ consists of two (overlapping) parts, namely points $x$ with $x_N=0$ and
points, where at least two coordinates are equal.
Assume now that $x(0)=(x(0)_1,\ldots,x(0)_N) $ satisfies $x(0)_N=\ldots=x(0)_{N-l+1}=0$ and 
 $x(0)_{N-l}>0$ with some $l=1,\ldots,N$. Then $\tilde e_N(x(0))=\ldots=\tilde e_{N-l+1}(x(0))=0$ and 
$\tilde e_{N-l}(x(0))>0$ by the form of the elementary symmetric polynomials. This, $\nu>0$, 
(\ref{elementary-symm-b-1-k}), and the Taylor expansion  now imply that $\tilde e_N(x(t))>0$ for  $t\in]0,t_0]$ and
a suitable  $t_0>0$. This means that our assumption
$\tilde e(t)\not\in e(W_N^B )$ for $t\in[0,t_0]$ is caused by the fact that
 at least two coordinates of $x(t)$ are equal for  $t\in[0,t_0]$.
>From this we conclude as in the proof of Theorem \ref{ode-ex-unique-a-thm}
 that 
$\tilde e(t)\not\in e(W_N^B)$ holds for all $t\ge0$.
With this observation the proof can be completed precisely as  the proof of Theorem \ref{ode-ex-unique-a-thm} where one has to use
the fact that  the solution of  (\ref{ODE-b}) with start in $0$ is given by
$x(t):=\sqrt{2t}\cdot y$ for the vector $y$ from Lemma \ref{special-solution-B1}.
\end{proof}

A slightly more complicated variant of Theorem \ref{ode-ex-unique-b-thm} can be stated for the case $\nu=0$ under some restriction.
This result will be a consequence of the study of the root system $D_N$ in the end of the next section.
We finally observe that 
Theorem \ref{ode-ex-unique-b-thm} can be supplemented by an analogue of Lemma 
\ref{lower-boundary-a} with the same proof.

\section{The root system $D_N$}

For the root systems $D_N$, we consider the associated Weyl chamber
$$C_N^D:=\{x\in\mathbb R^N: \quad x_1\ge \ldots\ge x_{N-1}\ge |x_N|\}\subset \mathbb R^N$$
as well as its interior $ W_N^B$. The ODE (\ref{ODE-general})  has in this case the form
\begin{equation}\label{ODE-d}
\frac{dx(t)}{dt}= H_D(x(t)) \quad\text{with}
 \quad H_D(x):=\left(\begin{matrix}\sum_{j\ne1}\Bigl( \frac{1}{x_1-x_j}  +
 \frac{1}{x_1+x_j}\Bigr)\\ \vdots\\ \sum_{j\ne N}
 \Bigl(\frac{1}{x_N-x_j}+   \frac{1}{x_N+x_j}\Bigr)\end{matrix}\right)
\end{equation}
on  $ W_N^B$. Similar to the preceding cases, we have by  Lemma 4.1 of \cite{AV1}:

\begin{lemma}\label{deterministic-boundary-D}
 For $\epsilon>0$ consider the open subset
 $U_\epsilon:=\{x\in C_N^D:\> d(x,\partial C_N^D)>\epsilon\}$.
of $C_N^D$. Then
 for each starting point $x_0\in U_\epsilon$, the ODE (\ref{ODE-d}) with $x(0)=x_0$ admits a unique solution. 
This solution exists for all $t>0$  with $x(t)\in U_\epsilon$.
\end{lemma}

We next recapitulate some facts on  Laguerre polynomials and
 proceed as in Section 4 of \cite{AV1}.
Using the  representation
$$L_N^{(\alpha)}(x):=
\sum_{k=0}^N { N+\alpha\choose N-k}\frac{(-x)^k}{k!} \quad\quad(\alpha\in\mathbb R, \> N\in\mathbb N)$$
of the Laguerre polynomials according to (5.1.6) of Szeg\"o \cite{S}, we  form the polynomial 
 $L_N^{(-1)}$ of order $N\ge1$
where,
 by (5.2.1) of \cite{S}, 
\begin{equation}\label{laguerre-1}
L_N^{(-1)}(x)=-\frac{x}{N}L_{N-1}^{(1)}(x).
\end{equation}
We denote the $N$ ordered zeros of  $L_N^{(-1)}$ by  $z_1>\ldots>z_{N-1}>z_N=0$.
Similar to the preceding cases, we obtain from Section 4 of \cite{AV1}:

\begin{lemma}\label{special-solution-D}
Let $y\in C_N^D$ be the vector with
\begin{equation}\label{y-max-D}
2\cdot (z_1,\ldots, z_{N-1},z_N=0)= (y_1^2,\ldots,y_N^2).
\end{equation}
 Then for each $c>0$,  $x(t):=\sqrt{t+c^2}\cdot y $ is a solution of (\ref{ODE-d}). 
\end{lemma}

Again, the growth of these particular solutions is typical; see \cite{VW}:

\begin{lemma}\label{growth-d} For each solution $x$ of  (\ref{ODE-d}) with start in  $x(0)\in W_N^D$,
\begin{equation}\label{ODE-norm-d}
\|x(t)\|^2 =2N(N-1)t+\|x(0)\|^2.
\end{equation}
\end{lemma}

We again have a stability result for the special solutions in Lemma \ref{special-solution-D}.
Its proof is completely analog to that of Lemma
\ref{stabilily-lemma-b}; we  omit the proof.

\begin{lemma}\label{stabilily-lemma-d}  
For each  $x_0\in W_N^D$, the solution $x$ of (\ref{ODE-d})   has the form
$$x(t)= \sqrt{ 2N(N-1)t+\|x_0\|^2}\cdot \phi(t)  \quad\quad(t\ge0)$$
where  $\phi$ satisfies
$$\|\phi(t)\|=1 \quad\quad\text{and}\quad\quad \lim_{t\to\infty} \phi(t) = \frac{2}{N(N-1)} y $$
with the vector $y$ from Lemma \ref{special-solution-D}.
\end{lemma}

Beside the particular solutions in  Lemma \ref{special-solution-D} 
we  here have the following  observation. This result fits with Eq.~(\ref{laguerre-1})
for $L_N^{(-1)}$.

\begin{lemma}\label{special-lemma-d}
Let  $x_0=(x_{0,1},\ldots,x_{0,N}) \in W_N^D$ with $x_{0,N}=0$. Then the associated solution $x$
of (\ref{ODE-d})   satisfies $x(t)_N=0$ for all $t$,
 and the first $N-1$ components
 $(x(t)_1,\ldots,x(t)_{N-1})$ solve the ODE  (\ref{ODE-b}) with dimension $N-1$ and  $\nu=2$.

Moreover, if $x_{0,N}>0$ or $<0$, then for all $t$,  $x(t)_N>0$ or $<0$ respectively.
\end{lemma}

\begin{proof}
If $x_{0,N}=0$, then by (\ref{ODE-d}), $\frac{d}{dt}x(t)_N=0$. This shows the first statements.
This and the fact that the curves $(x(t))_t$ are either equal or do not intersect then show the second statement.
\end{proof}

It is possible also to derive a result which is analog to Lemma \ref{symmetric-pol-pol-in-t-b}.
We skip this statement and proceed directly to the following analogue of Theorem \ref{ode-ex-unique-b-thm}

\begin{theorem}\label{ode-ex-unique-d-thm}
Let $N\ge 2$.
Then, for each starting point $x_0\in C_N^D$, the ODE  (\ref{ODE-d})
has a unique solution for all $t\ge0$ in the sense of Theorem \ref{ode-ex-unique-a-thm}.
\end{theorem}

\begin{proof} The proof is similar to that of Theorem \ref{ode-ex-unique-b-thm}, and we
keep the notations from Theorem \ref{ode-ex-unique-a-thm}.
For  $x_0\in  W_N^D$, the assertion is clear.
Now let $x_0=(x_1(0),\ldots,x_N(0))\in\partial C_N^D$. We transform
 (\ref{ODE-d}) as in Theorem \ref{ode-ex-unique-b-thm}  for $\nu=0$
and obtain
that for a solution $(x(t))_{t\ge0}$ of  (\ref{ODE-d}), the function
$\tilde e(t):= (\tilde e_1^N(x(t)),\ldots,\tilde e_1^N(x(t)) )$ satisfies
\begin{equation}\label{elementary-symm-d-1}
\frac{d}{dt}\Bigl(\tilde e_k^N(x(t))\Bigr)= 
2 (N-k+1)(N-k)\tilde e_{k-1}^{N}(x(t)) \quad\quad(k\ge2)
\end{equation}
and
\begin{equation}\label{symm-eq-d-1}
\frac{d}{dt}\Bigl(\tilde e_1^{N}(x(t))\Bigr)= 2N(N-1).
\end{equation}
In particular, for $k=N$,
\begin{equation}\label{symm-eq-d-N}
\frac{d}{dt}\Bigl(\tilde e_N^{N}(x(t))\Bigr)= 0.
\end{equation}
We now consider different cases. If $x_N(0)=0$, then we obtain from (\ref{symm-eq-d-N})
 that $\tilde e_N^{N}(x(t))=0$
and thus  $x_N(t)=0$  for all $t\ge0$. If we insert the trivial component $x_N(t)=0$  for $t\ge0$ 
into our ODE
 (\ref{ODE-d}), we get an ODE of the form (\ref{ODE-b}) of type B in $N-1$ dimensions with some $\nu>0$. 
Therefore, in this case,
 Theorem \ref{ode-ex-unique-d-thm} follows from Theorem \ref{ode-ex-unique-b-thm}.
Assume now that  $x_N(0)\ne0$. Here, the theorem can be proved in the same way as Theorem \ref{ode-ex-unique-a-thm}
 where one has to use
the fact that  the solution of   (\ref{ODE-d}) with start in $0$ is given by
$x(t):=\sqrt{2t}\cdot(\sqrt y_1,\ldots,\sqrt y_{N-1},0)$. 
\end{proof}

There is also an analogue of Lemma 
\ref{lower-boundary-a} for the case D with the same proof.

Clearly, the ODE (\ref{ODE-d}) of type D is closely related with the ODE (\ref{ODE-b}) of type B for $\nu=0$ by
Lemma \ref{special-lemma-d}. In particular, we obtain with Theorem \ref{ode-ex-unique-d-thm}:

\begin{corollary}\label{ode-ex-unique-b-thm-nu-null}
Let $N\ge 2$ and $\nu=0$ for the root system $B_N$.
Then, for each starting point $x_0=(x_1(0),\ldots,x_N(0))\in C_N^B$ with $x_N(0)>0$, the ODE  (\ref{ODE-b})
has a unique solution $x$ for all $t\ge0$ in the sense of Theorem \ref{ode-ex-unique-a-thm} with 
$x(t)\in W_N^B$ for $t>0$. On the other hand, if  $x_N(0)=0$, then there is no solution  $x$ of (\ref{ODE-b})
 with $x(t)\in W_N^B$ for $t>0$.
\end{corollary}

\end{document}